\documentclass[12pt]{amsart}

\usepackage{latexsym, amsmath, amssymb, amsfonts, amscd, amsthm, mathrsfs}
\usepackage[all, cmtip]{xy}

\newcommand{\Cb}{{\mathbb C}}
\newcommand{\Rb}{{\mathbb R}}

\newcommand{\Tb}{{\mathbb T}}

\newcommand{\TM}{{\rm TM}}
\newcommand{\rM}{{\rm M}}

\newcommand{\pa}{\|}

\newtheorem{theorem}{Theorem}[section]
\newtheorem{corollary}[theorem]{Corollary}

\newtheorem{proposition}[theorem]{Proposition}

\theoremstyle{definition}

\newtheorem{example}[theorem]{Example}

\begin{document}

\title{Smooth Approximation of Lipschitz Projections}
\author{Hanfeng Li}
\thanks{Partially supported by NSF Grant DMS-0701414.}
\address{Department of Mathematics \\
SUNY at Buffalo \\
Buffalo, NY 14260-2900, U.S.A.} \email{hfli@math.buffalo.edu}
\date{October 14, 2009}


\begin{abstract}
We show that any Lipschitz projection-valued function $p$ on a connected closed Riemannian manifold
can be approximated uniformly by smooth projection-valued functions $q$ with Lipschitz constant close to that of $p$.
This answers a question of Rieffel.
\end{abstract}

\maketitle


\section{Introduction} \label{introduction:sec}

The question of approximating continuous functions on Riemannian manifolds by
smooth functions preserving geometric properties has a long history. In \cite{GW72, GW74, GW79}
Greene and Wu studied such questions for real-valued functions and geometric properties
such as having Lipschitz constant bounded above by a fixed number, and used such results to obtain geometric applications
\cite{GW74}.

In a previous version of \cite{Rieffel06} Rieffel asked the question whether, for any Lipschitz function $p$ on a compact
Riemannian manifold $\rM$ with values in the projection set of a matrix algebra $M_n(\Cb)$ and any $\varepsilon>0$, there is
a smooth function $q$ on $\rM$ also with values in the projection set of $M_n(\Cb)$ with $\pa p-q\pa_{\infty}<\varepsilon$
and $L(q)<L(p)+\varepsilon$. Here $\pa f\pa_{\infty}$ denotes the supremum norm of $f$ and
\begin{eqnarray}  \label{L:eq}
L(f)=\sup_{x, y\in \rM, \, x\neq y}\frac{\pa f(x)-f(y)\pa}{\rho(x, y)}
\end{eqnarray}
denotes the Lipschitz constant of $f$, for $\rho$ denoting the geodesic distance on $\rM$. An affirmative answer to this question
has direct application on obtaining lower bounds of the Lipschitz constants for projection-valued functions on $\rM$
representing a fixed vector bundle on $\rM$.

In this note we answer Rieffel's question affirmatively. In fact, we shall deal more generally with
functions with values in the projection set of any $C^*$-algebra.
Following \cite{Rieffel06}, by a {\it real $C^*$-subring} we mean a
norm-closed $*$-subring of a $C^*$-algebra which is closed under
multiplication by scalars in $\Rb$.
For example, $M_n(\Rb)$ is a real $C^*$-subring contained in $M_n(\Cb)$.
For
a compact manifold $\rM$ and a real $C^*$-subring $A$, denote by $C(\rM, A)$ the real $C^*$-subring of all
continuous $A$-valued functions on $\rM$, and denote by $C^{\infty}(\rM, A)$ the subalgebra of all
smooth $A$-valued functions on $\rM$.

\begin{theorem} \label{approximation:thm}
Let $\rM$ be a connected closed Riemannian manifold. Let $A$ be a real $C^*$-subring.
For any projection $p\in C(\rM, A)$ and any $\varepsilon>0$, there exists a
projection $q\in C^{\infty}(\rM, A)$ with $\pa p-q\pa_{\infty}<\varepsilon$
and $L(q)<L(p)+\varepsilon$.
\end{theorem}

As a sample, Rieffel has obtained 
lower bound 
of the Lipschitz constants 
for  projections representing certain complex line bundle
on the two-torus $\Tb^2$ in \cite[Proposition 12.1]{Rieffel06}, via obtaining such 
bound for smooth projections first and then extend the bound to nonsmooth projections using
Theorem~\ref{approximation:thm}. 

The proof of Theorem~\ref{approximation:thm} has two ingredients. The first one is
that part of Greene and Wu's results hold generally for functions valued in any Banach space.
We recall these results in Section~\ref{approximation:sec}.
The second one is a fine estimate of the Lipschitz seminorm of the projection in a $C^*$-algebra
obtained in the usual way from an ``almost-projection''. We give this estimate and prove Theorem~\ref{approximation:thm} in Section~\ref{estimate:sec}.

I am grateful to Wei Wu for comments.

\section{Smooth Approximation of Continuous Functions} \label{approximation:sec}

In \cite{GW72} Greene and Wu introduced {\it the Riemannian convolution smoothing process}
for approximating continuous real-valued functions on a Riemannian manifold by smoothing functions
preserving geometric properties. They used this technique to obtain certain geometric application
\cite{GW74}. In fact, much of Greene and Wu's results work for functions valued in a Banach space.
Let us recall their construction on \cite[Page 646-647]{GW72}.

Fix a positive integer $n$. Let $\kappa:\Rb\rightarrow \Rb$ be a non-negative smooth function which has support contained in the interval
$[-1, 1]$, is constant in a neighborhood of $0$, and satisfies $\int_{v\in \Rb^n}\kappa(\pa v\pa)=1$, where $\Rb^n$ is equipped
with the standard Lebesgue measure.
Let $V$ be a real Banach space.
Let $\rM$ be a Riemannian manifold of dimension $n$ without boundary, and let $K$ be a compact subset of $\rM$.
Then there is an $\varepsilon_K>0$ such that, for every $x$ in some neighborhood of $K$ and every $v$ in the tangent space
$\TM_x$ of $\rM$ at $x$ with $\pa v\pa<\varepsilon_K$, the exponential $\exp_xv$ is defined.
For any continuous function $f: \rM\rightarrow V$ and any $0<\varepsilon<\varepsilon_K$ define
a $V$-valued function $f_{\varepsilon}$ on a neighborhood of $K$ by
\begin{eqnarray*}
f_{\varepsilon}(x)=\frac{1}{\varepsilon^n}\int_{v\in \TM_x}\kappa(\frac{\pa v\pa}{\varepsilon})f(\exp_xv),
\end{eqnarray*}
where $\TM_x$ is equipped with the standard Lebesgue measure via a linear isometry between $\TM_x$ and the Euclidean space $\Rb^n$.

In \cite[Page 647]{GW72} and \cite[Lemma 8]{GW74} Greene and Wu proved the following theorem for the case $V=\Rb$.
Their proof works for any Banach space  $V$.

\begin{theorem}[Greene-Wu] \label{smoothing:thm}
Let $f: \rM\rightarrow V$ be continuous. Let $K$ be a compact subset of $\rM$.
When $\varepsilon>0$ is small enough (depending only
on $\rM$ and $K$), $f_{\varepsilon}$ is
a smooth function on an open neighborhood of $K$.
The function $f_{\varepsilon}$ converges to $f$ uniformly on $K$ as $\varepsilon\to 0$.
If $f$ is smooth on a neighborhood of $K$, then $f_{\varepsilon}$
converges to $f$ in the smooth topology on $K$ as $\varepsilon\to 0$.
If $\rM$ is connected and $f$ has Lipschitz constant $D$ on a neighborhood $W$ of $K$,
i.e.,
$$ \sup_{x, y\in W,\, x\neq y}\frac{\pa f(x)-f(y)\pa}{\rho(x, y)}=D,$$
where $\rho$ denotes the geodesic distance on $\rM$, then for any $\delta>0$, when $\varepsilon$ is small enough
one has $\pa \partial_v(f_{\varepsilon})\pa<D+\delta$  for all $x\in K$ and unit vector $v\in \TM_x$.
\end{theorem}

One direct consequence of the above approximation theorem is

\begin{corollary} \label{compact smoothing:cor}
Let $\rM$ is a connected closed Riemannian manifold. For any continuous map $f:\rM\rightarrow V$
and any $\delta>0$, when $\varepsilon>0$ is small enough, $f_{\varepsilon}$ is a
smooth function on $\rM$ with $\pa f-f_{\varepsilon}\pa_{\infty} <\delta$ and $L(f_{\varepsilon})<L(f)+\delta$,
where $L(f_{\varepsilon})$ and $L(f)$ are defined by (\ref{L:eq}).
\end{corollary}

Using Theorem~\ref{smoothing:thm} for the case $V=\Rb$, Greene and Wu actually showed \cite[Proposition 2.1]{GW79} that for
any connected (possibly noncompact) Riemannian manifold $\rM$ without boundary, any Lipschitz function $f:\rM\rightarrow \Rb$,
any continuous function $g:\rM\rightarrow \Rb_{>0}$, and any $\delta>0$,
there is some smooth function $h:\rM\rightarrow \Rb$ such that $|f-h|< g$ on $\rM$ and $L(h)<L(f)+\delta$.
It would be interesting to see whether this result extends to every Banach space $V$.

\section{Estimate of Lipschitz Constants for Projections} \label{estimate:sec}

The key in our approach is the following estimate of the Lipschitz seminorm of
the projection obtained from an ``almost-projection'' in the usual way.

\begin{proposition} \label{approx proj to projection:prop}
Let $A$ be  a unital $C^*$-algebra, and let $a\in A$ be self-adjoint whose
spectrum has empty intersection with the interval $(\delta, 1-\delta)$ in $\Rb$
for some $0<\delta<1/2$. Let $L$ be a (possibly $+\infty$-valued) seminorm
on $A$. Suppose that $L$ is lower semi-continuous, vanishes on $1_A$, and $L(b^{-1})\le \pa b^{-1}\pa^2 L(b)$
for every invertible $b\in A$.  Let $f$ denote the characteristic function of the interval
$[1-\delta, +\infty)$ on $\Rb$.
Then
$$L(f(a))\le L(a)/(1-2\delta),$$
where $f(a)$ denotes the continuous-functional calculus of $a$ under $f$.
\end{proposition}
\begin{proof}
Take $R>0$ such that the spectrum of $a$ is contained in the union
of the intervals $(-\infty, \delta]$
and $[1-\delta, 1+R]$ in $\Rb$.

Denote by $g$ the characteristic function of $\{z\in \Cb: {\rm Re} z\ge (\delta+1/2)/2\}$ on
$\Cb$.
Note that $g$ is holomorphic on $\{z\in \Cb: {\rm  Re}z\neq (\delta+1/2)/2\}$.
For each $s>R$, denote by $\gamma_s$
the
rectangle with vertices
$1/2+si$, $1/2-si$,
$(1+s)-si$, and $(1+s)+si$, parameterized as a piecewise smooth simple closed curve
in the anti-clockwise direction.
Then
\begin{eqnarray*}
L(f(a))&=&L(g(a))=L(\frac{1}{2\pi i}\int_{\gamma_s}g(z)(z-a)^{-1}\, dz)\\
&\le & \frac{1}{2\pi}\int_{\gamma_s}L(g(z)(z-a)^{-1})\, d|z|\\
&=& \frac{1}{2\pi}\int_{\gamma_s}L((z-a)^{-1})\, d|z|\\
&\le & \frac{1}{2\pi}\int_{\gamma_s}\pa (z-a)^{-1}\pa^2 L(z-a)\, d|z|\\
&=& \frac{L(a)}{2\pi} \int_{\gamma_s}\pa (z-a)^{-1}\pa^2 \, d|z|,
\end{eqnarray*}
where the first inequality follows from the lower semi-continuity of $L$.
For $z$ in the line segment from $1/2-si$ ($(1+s)+si$ resp.) to $(1+s)-si$ ($1/2+si$ resp.),
one has $\pa (z-a)^{-1}\pa^2\le s^{-2}$ since $a$ is self-adjoint. For
$z=1/2+ti$ with $t\in \Rb$, one has $\pa (z-a)^{-1}\pa^2\le ((1/2-\delta)^2+t^2)^{-1}$.
  For
$z=(1+s)+ti$ with $t\in \Rb$, one has $\pa (z-a)^{-1}\pa^2\le ((s-R)^2+t^2)^{-1}$.
Therefore
\begin{eqnarray*}
L(f(a))&\le & \frac{L(a)}{2\pi} \int_{\gamma_s}\pa (z-a)^{-1}\pa^2 \, d|z|\\
&\le & \frac{L(a)}{2\pi}(2s^{-2}(1/2+s)+\int^{+\infty}_{-\infty}((1/2-\delta)^2+t^2)^{-1}\, dt\\
& &+\int^{+\infty}_{-\infty}((s-R)^2+t^2)^{-1}\, dt)\\
&=& \frac{L(a)}{2\pi}(s^{-2}(1+2s)+\frac{1}{1/2-\delta}\int^{+\infty}_{-\infty}(1+t^2)^{-1}\, dt\\
& &+\frac{1}{s-R}\int^{+\infty}_{-\infty}(1+t^2)^{-1}\, dt)\\
&=& \frac{L(a)}{2\pi}(s^{-2}(1+2s)+\frac{\pi}{1/2-\delta}+\frac{\pi}{s-R}).
\end{eqnarray*}
Letting $s\to \infty$, we obtain $L(f(a))\le L(a)/(1-2\delta)$ as desired.
\end{proof}

Using Proposition~\ref{approx proj to projection:prop}, Rieffel has  improved 
various estimates in a previous version of Sections 3, 4, 6 and 10 of \cite{Rieffel06}.

In general the estimate in Proposition~\ref{approx proj to projection:prop} is the best possible, as the following
example shows.

\begin{example} \label{estimate:example}
Let $X=\{x_1, x_2\}$ be a $2$-point metric space with the metric $\rho(x_1, x_2)=1$. Define $L$ on $C(X)$ via
(\ref{L:eq}).
By \cite[Proposition 2.2]{Rieffel06} the seminorm $L$ satisfies the conditions in Proposition~\ref{approx proj to projection:prop}.
For $0<\delta<1/2$ define $a\in C(X)$ by
$a(x_1)=\delta$ and $a(x_2)=1-\delta$. Then $(f(a))(x_1)=0$ and $(f(a))(x_2)=1$, where $f$ is as in  Proposition~\ref{approx proj to projection:prop}.
Thus
$L(a)=1-2\delta$ and $L(f(a))=1=L(a)/(1-2\delta)$.
\end{example}

We are ready to prove Theorem~\ref{approximation:thm}.

\begin{proof}[Proof of Theorem~\ref{approximation:thm}]
Let $0<\delta<1/2$. By Corollary~\ref{compact smoothing:cor} applied to
$V$ equal to the self-adjoint part of $A$ and $f$ equal to $p$, we can find some
self-adjoint $p_1\in C^{\infty}(\rM, A)$ with $\pa p-p_1\pa_{\infty}<\delta$ and $L(p_1)<L(p)+\varepsilon/2$.
Since $p$ is a projection, by \cite[Lemma 3.2]{Rieffel06} the spectrum of $p_1$ is contained in
the union of the intervals $[-\delta, \delta]$ and $[1-\delta, 1+\delta]$ in $\Rb$.
Let $f$ denote the characteristic function of the interval
$[1-\delta, +\infty)$ on $\Rb$.
Set $q=f(p_1)$.

Say, $A$ is a real $C^*$-subring contained in
a $C^*$-algebra $B$.
Since $C^{\infty}(\rM, B)$ is closed under the holomorphic functional
calculus, we have $q\in C^{\infty}(\rM, B)$.
Using polynomial approximations it is easy to see that the self-adjoint part of
$C(\rM, A)$ is closed under continuous functional calculus for
real-valued continuous functions on $\Rb$ vanishing at $0$ (see for example the proof of \cite[Proposition 2.4]{Rieffel06}).
Therefore $q\in C^{\infty}(\rM, B)\cap C(\rM, A)=C^{\infty}(\rM, A)$.

One has
$$ \pa q-p\pa_{\infty}\le \pa q-p_1\pa_{\infty}+\pa p_1-p\pa_{\infty}<\delta+\delta=2\delta.$$
It is readily checked that the Lipschitz seminorm $L$
on the $C^*$-algebra $C(\rM, B)$ defined by (\ref{L:eq})
satisfies the conditions in Proposition~\ref{approx proj to projection:prop}
(see for example the proof of \cite[Proposition 2.2]{Rieffel06}).
By Proposition~\ref{approx proj to projection:prop}
one has
$$ L(q)\le L(p_1)/(1-2\delta)<(L(p)+\varepsilon/2)/(1-2\delta).$$
Thus, when $\delta>0$ is small enough, we have $\pa q-p\pa_{\infty}<\varepsilon$ and $L(q)<L(p)+\varepsilon$ as desired.
\end{proof}



\begin{thebibliography}{10}

\bibitem{GW72}
R. E. Greene and H. Wu. On the subharmonicity and plurisubharmonicity of geodesically convex functions.
{\it Indiana Univ. Math. J.}  {\bf 22}  (1972/73), 641--653.

\bibitem{GW74}
R. E. Greene and H. Wu. Integrals of subharmonic functions on manifolds of nonnegative curvature.
{\it Invent. Math.}  {\bf 27}  (1974), 265--298.

\bibitem{GW79}
R. E. Greene and H. Wu. $C^{\infty}$ approximations of convex, subharmonic, and plurisubharmonic functions.
{\it Ann. Sci. \'{E}cole Norm. Sup. (4)}  {\bf 12}  (1979), no. 1, 47--84.

\bibitem{Rieffel00}
M. A. Rieffel. Gromov-Hausdorff distance for
quantum metric spaces.
{\it Mem. Amer. Math. Soc.} {\bf 168} (2004), no. 796, 1--65.
arXiv:math.QA/0011063.

\bibitem{Rieffel06}
M. A. Rieffel. Vector bundles and Gromov-Hausdorff distance. {\it J. K-Theory} to appear. 
arXiv:math/0608266.

\end{thebibliography}
\end{document}